\documentclass[reqno]{amsart}
\pdfoutput=1

\usepackage{amsmath}
\usepackage{mathrsfs}
\usepackage{amssymb}
\usepackage{amsthm}
\usepackage[top=95pt,bottom=80pt,left=85pt,right=85pt,footskip=35pt]{geometry}

\usepackage{xcolor}
\usepackage{enumerate}
\usepackage[all]{xy}
\usepackage{tikz}
\usetikzlibrary{shapes,positioning,intersections,quotes,calc}

\colorlet{green}{green!50!black}

\theoremstyle{definition}
\newtheorem{Def}{Definition}

\newtheorem*{Main}{Main Theorem}
\newtheorem{Thm}{Theorem}[section]
\newtheorem{Prop}[Thm]{Proposition}
\newtheorem{Lemma}[Thm]{Lemma}
\newtheorem{Rmk}[Thm]{Remark}

\newcommand{\PP}{\mathbb{P}}
\newcommand{\Cc}{\mathcal{C}}
\newcommand{\Dd}{\mathcal{D}}
\newcommand{\Ee}{\mathcal{E}}
\newcommand{\Kk}{\mathcal{K}}
\newcommand{\Ll}{\mathcal{L}}

\newcommand{\Oo}{\mathcal{O}}
\newcommand{\bb}{\mathfrak{b}}

\newcommand{\kk}{\mathfrak{k}}

\newcommand{\tS}{X}
\newcommand{\Bl}{\operatorname{Bl}}
\newcommand{\Sym}{\operatorname{Sym}}

\title{Bigness of tangent bundles of blow-ups of ruled surfaces}

\author{Hosung Kim and Jeong-Seop Kim}

\address[Hosung Kim]{Department of Mathematics\\
Changwon National University\\
20 Changwondaehak-ro\\
Uichang-gu\\
Changwon-si\\
Gyeongsangnam-do\\
51140 Korea}
\email{hosungkim@changwon.ac.kr}

\address[Jeong-Seop Kim]{Department of Mathematics Education\\
Sunchon National University\\
255 Jungang-ro\\
Suncheon-si\\
Jeonnam-Gwangju\\
57922 Korea}
\email{jeongseop@scnu.ac.kr}

\begin{document}

\maketitle

\begin{abstract}
We determine when the bigness of the tangent bundle is preserved under point blow-ups of a ruled surface $S=\PP_C(\Ee)$, giving a sharp bound on the number of blow-up points in terms of the first Segre invariant $s_1(\Ee)$.
\end{abstract}

\section{Introduction}

Throughout this paper, all varieties are defined over the field of complex numbers.
Positivity properties of the tangent bundle impose strong restrictions on the geometry of a smooth projective variety.
Mori's solution to Hartshorne's conjecture characterizes projective space by the ampleness of~$T_X$ \cite{Mor79}, and the Campana--Peternell conjecture predicts an analogous rigidity for Fano manifolds with nef~$T_X$ \cite{CP91}.
Weaker positivity notions, especially bigness and pseudo-effectivity of $T_X$, have recently been studied from several viewpoints, including VMRT methods, Fano classifications, dynamical rigidity, and surface geometry; see \cite{Hsi15,Mal21,HLS22,HL23,SZ25,JLZ25,KKL25,KKL24}.
In this paper, we say that a vector bundle $V$ is big on $X$ if the tautological line bundle $\Oo_{\PP_X(V)}(1)$ is big on $\PP_X(V)$.

For a ruled surface $S=\PP_C(\Ee)$ over a smooth projective curve $C$, it was proved that $T_S$ is big if and only if $\Ee$ is unstable or $C=\PP^1$ \cite{Kim23}.
It is then natural to ask how this bigness behaves under blow-ups at points.
This is motivated by the known result for blow-ups of $\PP^2$ in general position, where bigness holds precisely for at most $4$ points \cite{HLS22}.
The purpose of this paper is to answer this question for ruled surfaces.

\begin{Main}
Let $C$ be a smooth projective curve, and let $S=\PP_C(\Ee)$ be a ruled surface over~$C$ with the first Segre invariant $s_1(\Ee)=-n$.
Let $\tS$ be the blow-up of $S$ at $m$ points \emph{in general position}.
If $C\simeq \PP^1$, then $T_{\tS}$ is big if and only if either $m\leq 3$ for $n=0,1$ or $m\leq n+1$ for $n\geq 2$.
Otherwise, if~$C\not\simeq\PP^1$, then $T_{\tS}$ is big if and only if $m\leq n-1$ and $n\geq 1$.
\end{Main}

The notion of \emph{general position} appearing in the statement above will be made precise in Definition~\ref{general-position}.
The proof of the theorem is given separately in Theorem~\ref{positive-genus} for the case $C\not\simeq \PP^1$ and in Theorem~\ref{genus-zero} for the case $C\simeq\PP^1$.
For the bigness assertion, the proof in the case $C\simeq\PP^1$ relies on explicit calculations of effective divisors on $\PP(T_{\tS})$ arising from families of rational curves, while in the case $C\not\simeq\PP^1$, it follows more directly from the bigness of the relative tangent bundle.
For the non-bigness assertion, we reduce the problem, via elementary transformations, to the cases treated in \cite{KKL24} when $C\simeq\PP^1$ and in \cite{Kim23} when $C\not\simeq\PP^1$.

\section{Preliminaries}

\subsection{Segre invariant}

In this section, we review the theory of ruled surfaces; for a standard reference, see \cite[Section V.2]{Har77}.
For a rank-two vector bundle $\Ee$ on $C$, \emph{the first Segre invariant} $s_1(\Ee)$ is defined by
\[
s_1(\Ee)
=\min\{\deg \Ee-2\deg M \mid M\subset \Ee \text{ is a line subbundle}\}.
\]
By the correspondence between line subbundles of $\Ee$ and sections of the ruled surface $\PP_C(\Ee)$,
\[
s_1(\Ee)=\min\{D^2\mid D\subset \PP_C(\Ee) \text{ is a section}\}.
\]
A section $C_0$ satisfying $C_0^2=s_1(\Ee)$ is called a minimal section.
For a rank-two vector bundle $\Ee$ on~$C$, $\Ee$ is unstable if and only if $s_1(\Ee)<0$.
In this case, the uniqueness of the minimal section follows from the negativity of its self-intersection number.
Indeed, if $C_0$ and $C_1$ are minimal sections, then $C_0\equiv C_1$, and hence $C_0.C_1=C_0^2=s_1(\Ee)<0$.
This contradicts the fact that two distinct irreducible curves have nonnegative intersection number.

\begin{Lemma}\label{unique-minimal-section}
Let $\Ee$ be a rank-two vector bundle on $C$ with $s_1(\Ee)<0$.
Then the minimal section $C_0$ on $\PP_C(\Ee)$ is unique.
\end{Lemma}

For the $n$-th Hirzebruch surface $\Sigma_n\simeq \PP_{\PP^1}(\Oo_{\PP^1}\oplus\Oo_{\PP^1}(-n))$ with $n>0$, the section corresponding to the line subbundle $\Oo_{\PP^1}\to \Oo_{\PP^1}\oplus\Oo_{\PP^1}(-n)$ is the unique minimal section $C_0$ of $\Sigma_n$.
It satisfies $C_0^2=-n$.

\begin{Def}\label{general-position}
Let $S=\PP_C(\Ee)$ be a ruled surface associated with an unstable vector bundle $\Ee$ over $C$.

If $s_1(\Ee)<0$, then we say that points $x_1,\ldots,x_m\in S$ are \emph{in general position} if none of them lies on $C_0$ and no two of them lie on the same fiber of $\pi:S\to C$.
Here, $C_0$ denotes the unique section with negative self-intersection $C_0^2=s_1(\Ee)=-n$.

If $s_1(\Ee)=0$ and $C\simeq \PP^1$, then we say that points $x_1,\ldots,x_m\in S$ are \emph{in general position} if no two of them lie on the same fiber of either of the two projections $S\to \PP^1$.
We remark that in this case $S\simeq \PP^1\times\PP^1$.
\end{Def}

In the theorems in the subsequent sections, the above definition specifies which configurations of points are allowed as centers of blow-ups of a ruled surface.


\begin{Lemma}\label{section-self-intersection}
Let $C$ be a smooth projective curve, and let $S=\PP_C(\Ee)$ be a ruled surface over $C$.
Assume that $s_1(\Ee)=-n<0$, and let $C_0$ be the minimal section of~$S$.
If $D$ is an irreducible section of~$S$, then either $D^2=-n$ or $D^2\geq n$.
Moreover, $D^2=-n$ if and only if $D=C_0$.
\end{Lemma}

\begin{proof}
Since $D$ is a section of $S$, there exists a divisor $\bb$ on $C$ such that $D\sim C_0+\bb f$.
Put $b=\deg \bb$.
Then $D^2=-n+2b$ and $D.C_0=-n+b$.
If $D=C_0$, then $D^2=-n$.
Assume now that $D\neq C_0$.
Since $D$ and $C_0$ are distinct irreducible curves, we have $D.C_0\geq 0$.
Thus $b\geq n$, and hence $D^2=-n+2b\geq n$.
\end{proof}

\subsection{Elementary transformation}

In this section, we recall elementary transformations of ruled surfaces.
In particular, we use elementary transformations to relate ruled surfaces with different Segre invariants.
For the general theory of elementary transformations, we refer to \cite{Mar82}.

Let $S=\PP_C(\Ee)$ be a ruled surface over a smooth curve $C$, and let $\pi:S\to C$ be the natural projection.
Let $x\in S$ be a point lying over $p\in C$, and let $f_p$ be the fiber of $\pi$ over $p$.
Let $\sigma:\tS\to S$ be the blow-up of $S$ at $x$.
The strict transform of $f_p$ on $\tS$ is a $(-1)$-curve, and hence it can be contracted.
We denote the resulting contraction by $\tau:\tS\to S'$.
Then $S'$ is again a ruled surface over $C$, say $S'=\PP_C(\Ee')$, and we have the following diagram.
\[
\xymatrix{
& \tS \ar[ld]_{\tau} \ar[rd]^{\sigma} & \\
S'=\PP_C(\Ee') && S=\PP_C(\Ee).
}
\]\
This operation is called \emph{the elementary transformation} of $S$ at $x$.

Let $D$ be an irreducible section of $\pi: S\to C$, and let $D'$ be the strict transform of $D$ on $S'$.
Then the self-intersection number of $D'$ is given by
\[
D'^2=
\begin{cases}
D^2-1 & \text{if $x\in D$},\\
D^2+1 & \text{if $x\notin D$}.
\end{cases}
\]

More generally, we can perform elementary transformations at several points simultaneously.
Let $x_1,\ldots,x_m\in S$ be points lying over distinct fibers of $\pi: S\to C$.
We blow up $S$ at $x_1,\ldots,x_m$ and then contract the strict transforms of the fibers over $p_1,\ldots,p_m$.
The resulting surface is again a ruled surface over $C$, which we denote by $S'=\PP_C(\Ee')$.
Let $D$ be a section of $S$, and let $D'$ be its strict transform on $S'$.
If $k$ denotes the number of points among $x_1,\ldots,x_m$ lying on $D$, then
\[
D'^2=D^2+m-2k.
\]
Indeed, each elementary transformation decreases the self-intersection by $1$ when the center lies on the section, and increases it by $1$ otherwise.

Now assume that $\Ee$ is unstable.
By Lemma~\ref{unique-minimal-section}, the ruled surface $S=\PP_C(\Ee)$ has a unique section $C_0$ with negative self-intersection.
With our convention, we have $C_0^2=s_1(\Ee)$.
If all centers $x_1,\ldots,x_m$ lie on $C_0$, then the corresponding section $\overline{C_0}$ on $S'$ satisfies $\overline{C_0}^2=C_0^2-m$.
It follows that $s_1(\Ee')=s_1(\Ee)-m$.

\begin{center}
\begin{figure}[h]
\begin{tikzpicture}[scale=0.87]

\pgfmathsetmacro{\wi}{3.5}
\pgfmathsetmacro{\he}{2.5}
\pgfmathsetmacro{\ri}{1.2}
\pgfmathsetmacro{\up}{2}
\pgfmathsetmacro{\t}{0.2}

\coordinate (a) at (0*\ri*\wi,0);
\coordinate (b) at (1*\ri*\wi,\up);
\coordinate (c) at (2*\ri*\wi,0);

\tikzset{>=latex}
\draw[thick,<-]
(0*\ri*\wi+\wi/2+\t,  \he+\t) --
(1*\ri*\wi-\t,     \he/2+\up+\t);
\draw[thick,->]
(1*\ri*\wi+\wi+\t, \he/2+\up+\t) --
(2*\ri*\wi+\wi/2-\t,     \he+\t);

\def\surface
{(0,0)--(\wi,0)--(\wi,\he)--(0,\he)--cycle;}

\begin{scope}[shift={($(a)$)}]
\foreach \i in {1,...,4} {
\draw[black,densely dashed]
(\i*\wi/5,0)--
(\i*\wi/5,\he);
}

\draw[very thick] (0.3,\he/2)--(\wi-0.3,\he/2);
\node[circle,fill=red,inner sep=0pt,minimum size=3pt] at (\wi/5,\he/2) {};
\node[circle,fill=red,inner sep=0pt,minimum size=3pt] at (2*\wi/5,\he/2) {};
\node[circle,fill=red,inner sep=0pt,minimum size=3pt] at (3*\wi/5,\he/2) {};
\node[circle,fill=red,inner sep=0pt,minimum size=3pt] at (4*\wi/5,\he/2) {};
\node[circle,fill=blue,inner sep=0pt,minimum size=3pt] at (3.5*\wi/5,0.5+\he/2) {};
\node[circle,fill=green,inner sep=0pt,minimum size=3pt] at (1.5*\wi/5,-0.7+\he/2) {};

\draw[thick] \surface;

\node at (\wi/2,-0.4) {\small $S'=\PP(\Ee')$};
\node at (0.35,\he/2-0.4) {\small $\overline{C_0}$};
\end{scope}

\begin{scope}[shift={($(b)$)}]

\draw[very thick] (0.3,\he-0.7)--(\wi-0.3,\he-0.7);
\foreach \i in {1,...,4} {
\draw[red,densely dashed]
(\i*\wi/5,\he) .. controls
(\i*\wi/5,\he/3*2-0.1) ..
(\i*\wi/5-0.2,\he/5*2-0.1);
\draw[black,densely dashed]
(\i*\wi/5,0) .. controls
(\i*\wi/5,\he/3-0.1) ..
(\i*\wi/5-0.2,\he/5*3-0.1);
}
\node[circle,fill=blue,inner sep=0pt,minimum size=3pt] at (3.5*\wi/5,0.9+\he/2) {};
\node[circle,fill=green,inner sep=0pt,minimum size=3pt] at (1.5*\wi/5,-0.7+\he/2) {};

\draw[thick] \surface;

\node at (\wi/2,-0.4) {\small $\tS$};
\end{scope}

\begin{scope}[shift={($(c)$)}]

\foreach \i in {1,...,4} {
\draw[red,densely dashed]
(\i*\wi/5,0)--
(\i*\wi/5,\he);
}
\node[circle,fill=blue,inner sep=0pt,minimum size=3pt] at (3.5*\wi/5,0.5+\he/2) {};
\node[circle,fill=green,inner sep=0pt,minimum size=3pt] at (1.5*\wi/5,-0.7+\he/2) {};

\node[circle,fill=black,inner sep=0pt,minimum size=3pt] at (1*\wi/5,\he/2-0.5) {};
\node[circle,fill=black,inner sep=0pt,minimum size=3pt] at (2*\wi/5,\he/2-0.3) {};
\node[circle,fill=black,inner sep=0pt,minimum size=3pt] at (3*\wi/5,\he/2-1) {};
\node[circle,fill=black,inner sep=0pt,minimum size=3pt] at (4*\wi/5,\he/2-0.7) {};

\draw[very thick] (0.3,\he/2)--(\wi-0.3,\he/2);
\draw[thick] \surface;

\node at (\wi-0.35,\he/2-0.4) {\small $C_0$};
\node at (\wi/2,-0.4) {\small $S=\PP(\Ee)$};
\end{scope}
\end{tikzpicture}\vspace{-1em}
\label{elem}
\caption{Elementary transformation between ruled surfaces with different $s_1$.}
\vspace{-2em}
\end{figure}
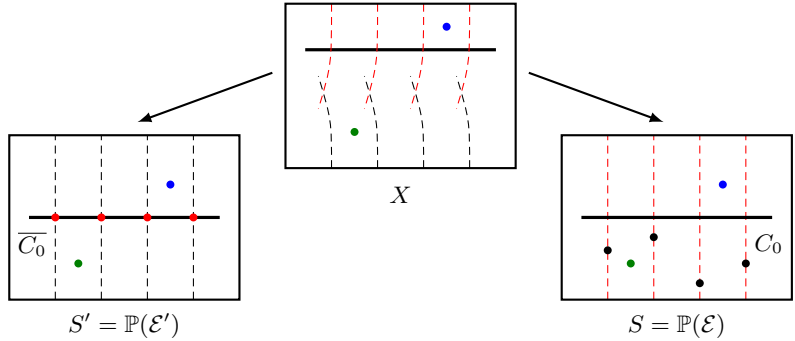
\end{center}

In the opposite direction, elementary transformations away from the negative section may increase the first Segre invariant.
We record the following consequence, which will be used later.

\begin{Lemma}\label{unstable-to-semistable}
Let $C$ be a smooth projective curve, and let $S=\PP_C(\Ee)$ be a ruled surface over $C$ such that $s_1(\Ee)=-n<0$.
If $S'=\PP_C(\Ee')$ is the ruled surface obtained from $S$ by taking elementary transformations at $n$ points $x_1,\,\ldots,\,x_n$ in general position, then $\Ee'$ is semistable.
\end{Lemma}

\begin{proof}
Suppose that $\Ee'$ is unstable.
Equivalently, $s_1(\Ee')<0$.
Then there exists an irreducible section $D'$ of $S'$ with $D'^2<0$.
Let $D$ be the strict transform of $D'$ on $S$ under the elementary transformations.
Let $k$ be the number of centers among $x_1,\ldots,x_n$ lying on $D$.
By the self-intersection formula for blow-ups and contractions, we have
\[
D'^2=D^2+n-2k.
\]
Since $D'^2<0$ and $k\leq n$, it follows that $D^2<2k-n\leq n$, and hence $D=C_0$ by Lemma~\ref{section-self-intersection}.

By the general-position assumption, none of the centers $x_1,\ldots,x_n$ lies on $C_0$.
Thus $k=0$, and the above formula gives $D'^2=C_0^2+n=0$, which contradicts $D'^2<0$.
Therefore, $\Ee'$ is semistable.
\end{proof}

\subsection{Criteria for the bigness of vector bundles}

Let $V$ be a vector bundle on a smooth projective variety $X$.
We say that $V$ is \emph{big} 
if the tautological line bundle $\Oo_{\PP(V)}(1)$ is big 
on $\PP(V)$.
We write $\zeta=c_1(\Oo_{\PP(V)}(1))$ and denote the natural projection by $\Pi:\PP_X(V)\to X$.
Throughout the remainder of this paper, we will mainly consider the case $V=T_X$.

\begin{Lemma}[{\cite[Lemma~6]{HLS22}}]\label{big_criterion}
Let $V$ be a vector bundle on a smooth projective variety $X$.
Then $V$ is big if and only if $k\zeta-\Pi^*H$ is pseudoeffective on $\PP_X(V)$ for some $k>0$ and some big divisor $H$ on $X$.
\end{Lemma}

In view of Lemma~2.4, the construction of effective divisors on $\PP_X(T_X)$ can be used to establish the bigness of $T_X$.
If $X$ is a smooth projective surface admitting a conic bundle structure $\pi: X\to C$, then the exact sequence
\[
0\to T_{X/C}\to T_X\to \pi^*T_C
\]
gives a nonzero global section of $T_X\otimes T_{X/C}^{\vee}$ that corresponds to the following effective divisor on $\PP(T_X)$ (see also \cite[Lemma~2.13]{HLS22}).
\begin{equation}\label{conic-bundle-formula}
\breve\Cc=\zeta+\Pi^*K_{X/C}.
\end{equation}
This divisor is called \emph{the total dual VMRT} associated to the family of conic fibers of $\pi: X\to C$.
For the detailed theory of total dual VMRTs, see \cite{HR04, OSW16}.

In particular, for a fibration structure $\pi: X\to B$, the bigness of $T_{X/B}$ implies the bigness of $T_X$ by the following criterion.

\begin{Lemma}\label{subbundle_lemma}
Let $W$ be a subbundle of a vector bundle $V$.
If $W$ is big, then $V$ is big.
\end{Lemma}

\begin{proof}
The inclusion $W\subset V$ induces an inclusion $\Sym^iW\otimes \Oo_X(-jH)\subset \Sym^iV\otimes \Oo_X(-jH)$.
Thus, if $H^0(X,\Sym^iW\otimes \Oo_X(-jH))\neq 0$, then $H^0(X,\Sym^iV\otimes \Oo_X(-jH))\neq 0$.
Therefore, the statement follows from Lemma~2.4 and \cite[Lemma~2.2]{HLS22}.
\end{proof}

We will also use the following criterion, which describes the behavior of the bigness of tangent bundles under a blow-up.

\begin{Lemma}[{\cite[Lemma 2.4]{HLS22}}]\label{blow-up}
Let $X$ and $Y$ be smooth projective varieties.
Assume that there is a birational morphism $X\to Y$.
If $T_X$ is big, then $T_Y$ is big.
\end{Lemma}

\section{Proof of the Main Theorem}

\begin{Prop}\label{big}
Let $S=\PP_C(\Ee)$ be the ruled surface associated with an unstable vector bundle $\Ee$ over a smooth projective curve $C$, and assume that $s_1(\Ee)=-n$.
If $\tS$ is the blow-up of $S$ at $m<n$ points lying on distinct fibers of $\pi: S\to C$, then $T_X$ is big.
\end{Prop}

\begin{proof}
After twisting $\Ee$ by a line bundle, we may assume that $\Ee=\pi_*\Oo_S(C_0)$.
Let $\mathfrak{e}=\det \Ee^\vee$, and let $E_i$ be the exceptional divisor over $x_i\in S$, where $x_i$ lies on the fiber over $p_i\in C$.
Then the relative tangent bundle of the conic bundle $\tS\to C$ is given by
\[
T_{\tS/C}\sim 2C_0+\pi^*\mathfrak{e}-E_1-E_2-\cdots-E_m.
\]
Note that
\[
T_{\tS/C}
\geq 2C_0+\pi^*(\mathfrak{e}-p_1-p_2-\cdots-p_m),
\]
and the divisor $\mathfrak{e}-p_1-\cdots-p_m$ is big on $C$, since $\deg(\mathfrak{e}-p_1-\cdots-p_m)=n-m>0$.
Thus $T_{\tS/C}$ is big by Lemma~\ref{big_criterion}, and hence $T_{\tS}$ is big by Lemma~\ref{subbundle_lemma}.
\end{proof}

Together with the converse, this yields the following theorem in the case $C\not\simeq \PP^1$.

\begin{Thm}\label{positive-genus}
Let $C$ be a smooth projective curve of genus $g>0$, and let $S=\PP_C(\Ee)$ be a ruled surface over $C$ such that $s_1(\Ee)=-n<0$.
If $\tS$ is the blow-up of $S$ at $m$ points in general position, then $T_{\tS}$ is big if and only if $m\leq n-1$. 
\end{Thm}

\begin{proof}
The bigness assertion has already been proved in Proposition~\ref{big}.
It remains to prove the non-bigness assertion.
By Lemma~\ref{blow-up}, it suffices to prove the assertion in the case $m=n$.

By the definition of general position, we may assume that the centers $x_1,\ldots,x_n$ of the blow-up $\tS\to S$ are in general position.
By contracting the strict transforms of the fibers passing through the points $x_i$ on $\tS$, we obtain another ruled surface $S'$ associated to a vector bundle $\Ee'$ over $C$; that is, $S'=\PP(\Ee')$.

By Lemma~\ref{unstable-to-semistable}, the vector bundle $\Ee'$ is semistable.
Hence, by \cite{Kim23}, the tangent bundle $T_{S'}$ of $S'=\PP(\Ee')$ is not big.
Since $\tS$ is also obtained as the blow-up $\tS\to S'$, Lemma~\ref{blow-up} implies that $T_{\tS}$ is not big.
\end{proof}

We next prove the assertion in the case $C\simeq\PP^1$; $\PP_C(\Ee)\cong \Sigma_n$ is the $n$-th Hirzebruch surface.
For the bigness assertion, Proposition~\ref{big} does not apply directly, since this case permits a larger number of blow-up points.
We instead carry out a detailed analysis of certain families of rational curves.

For the non-bigness assertion, we use the standard description of Hirzebruch surfaces in terms of elementary transformations: the surface $\Sigma_n$ is obtained from $\Sigma_1\simeq \operatorname{Bl}_p\PP^2$ by performing an elementary transformation at $n-1$ points lying on the minimal section, i.e., the canonical section.

Assume that $n\geq 2$ and $m\ge n-1$.
Let $X$ be the blow-up of $\Sigma_n$ at $m$ points in general position.
By reversing the above sequence of elementary transformations and then contracting the negative section of $\Sigma_1$, the surface $X$ can be realized as the blow-up of $\PP^2$ at $m-n+2$ closed points and $n-1$ infinitely near points over one of them.
Equivalently, there exist closed points $p_1,\ldots,p_{m-n+2}\in\PP^2$ and pairwise distinct first infinitely near points $q_1,\ldots,q_{n-1}$ over $p_1$, such that the collection $p_1,\ldots,p_{m-n+2},q_1,\ldots,q_{n-1}$ is in general position, and
\[
X \simeq \Bl_{p_1,\ldots,p_{m-n+2},q_1,\ldots,q_{n-1}}\PP^2
\]
as illustrated in the figure below.

\begin{center}
\begin{figure}[h]
\begin{tikzpicture}[scale=0.87]
\vspace{-0.5em}
\pgfmathsetmacro{\wi}{3.5}
\pgfmathsetmacro{\he}{2.5}
\pgfmathsetmacro{\ri}{1.2}
\pgfmathsetmacro{\up}{2}
\pgfmathsetmacro{\t}{0.2}

\coordinate (z) at (-\ri*1.5-\t,\he/2);
\coordinate (a) at (0*\ri*\wi,0);
\coordinate (c) at (1.55*\ri*\wi,0);

\tikzset{>=latex}
\draw[thick,<-]
(-\ri*1.5+\wi/5+\t, \he/2) -- (0-\t, \he/2);
\draw[thick,dotted,<->]
(3.5+\t, \he/2) -- (6.5-\t, \he/2);
\node at (5,\he/2+0.3) {\small elementary};
\node at (5,\he/2-0.3) {\small transformation};

\def\surface
{(0,0)--(\wi,0)--(\wi,\he)--(0,\he)--cycle;}
\def\bisection
{(\wi/2,\he/2-0.1) ellipse (1.2 and 0.5)}

\begin{scope}[shift={($(z)$)}]
\draw[thick,color=red,-stealth] (0,0)--(0.16,0.25);
\node at (0.30,0.35) {\small $q_1$};
\draw[thick,color=red,-stealth] (0,0)--(-0.25,0.16);
\node at (0.35,-0.30) {\small $q_2$};
\draw[thick,color=red,-stealth] (0,0)--(-0.25,-0.16);
\node at (-0.40,-0.25) {\small $q_3$};
\draw[thick,color=red,-stealth] (0,0)--(0.16,-0.25);
\node at (-0.40,0.25) {\small $q_4$};
\node[circle,fill=black,inner sep=0pt,minimum size=3pt] at (0,0) {};
\node at (0,-0.4) {\small $p_1$};
\node[circle,fill=blue,inner sep=0pt,minimum size=3pt] at (0.5,0.5) {};
\node at (0.5,0.73) {\small $p_4$};
\node[circle,fill=green,inner sep=0pt,minimum size=3pt] at (-0.3,-0.7) {};
\node at (-0.3,-0.95) {\small $p_2$};
\node[circle,fill=violet,inner sep=0pt,minimum size=3pt] at (0,0.8) {};
\node at (0,1.05) {\small $p_3$};
\draw[thick] (0,0) ellipse (\wi/4 and \he/2+0.3);

\node at (-0.85,-1.55) {\small $\PP^2$};
\end{scope}

\begin{scope}[shift={($(a)$)}]
\foreach \i in {1,...,4} {
\draw[black,densely dashed]
(\i*\wi/5,0)--
(\i*\wi/5,\he);
}

\draw[very thick] (0.3,\he/2)--(\wi-0.3,\he/2);
\node[circle,fill=red,inner sep=0pt,minimum size=3pt] at (\wi/5,\he/2) {};
\node[circle,fill=red,inner sep=0pt,minimum size=3pt] at (2*\wi/5,\he/2) {};
\node[circle,fill=red,inner sep=0pt,minimum size=3pt] at (3*\wi/5,\he/2) {};
\node[circle,fill=red,inner sep=0pt,minimum size=3pt] at (4*\wi/5,\he/2) {};
\node[circle,fill=blue,inner sep=0pt,minimum size=3pt] at (3.5*\wi/5,0.5+\he/2) {};
\node[circle,fill=green,inner sep=0pt,minimum size=3pt] at (1.5*\wi/5,-0.7+\he/2) {};
\node[circle,fill=violet,inner sep=0pt,minimum size=3pt] at (2.5*\wi/5,0.8+\he/2) {};

\draw[thick] \surface;
\node at (\wi/2,-0.4) {\small $S'=\Sigma_1$};
\end{scope}

\begin{scope}[shift={($(c)$)}]

\foreach \i in {1,...,4} {
\draw[red,densely dashed]
(\i*\wi/5,0)--
(\i*\wi/5,\he);
}
\node[circle,fill=blue,inner sep=0pt,minimum size=3pt] at (3.5*\wi/5,0.5+\he/2) {};
\node[circle,fill=green,inner sep=0pt,minimum size=3pt] at (1.5*\wi/5,-0.7+\he/2) {};
\node[circle,fill=violet,inner sep=0pt,minimum size=3pt] at (2.5*\wi/5,0.8+\he/2) {};

\node[circle,fill=black,inner sep=0pt,minimum size=3pt] at (1*\wi/5,\he/2-0.5) {};
\node[circle,fill=black,inner sep=0pt,minimum size=3pt] at (2*\wi/5,\he/2-0.3) {};
\node[circle,fill=black,inner sep=0pt,minimum size=3pt] at (3*\wi/5,\he/2-1) {};
\node[circle,fill=black,inner sep=0pt,minimum size=3pt] at (4*\wi/5,\he/2-0.7) {};

\draw[very thick] (0.3,\he/2)--(\wi-0.3,\he/2);
\draw[thick] \surface;

\node at (\wi/2,-0.4) {\small $S=\Sigma_5$};
\end{scope}
\end{tikzpicture}
\vspace{-0.5em}
\caption{Realizing the blow-up of $\Sigma_5$ at $7$ points as the blow-up of $\PP^2$ at $4$ closed points $p_1,p_2,p_3,p_4$ and $4$ infinitely near points $q_1,q_2,q_3,q_4$ over $p_1$.}
\vspace{-2em}
\end{figure}
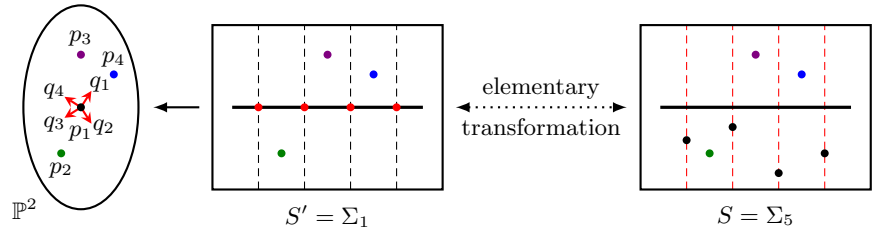
\end{center}

\begin{Rmk}\label{general-position-condition}
By the general-position condition in Definition~1, for every
$2\le i\le m-n+2$ and $1\le j\le n-1$, the point $p_i$ does not lie on the line through $p_1$ corresponding to the tangent direction determined
by $q_j$.
Moreover, no triple $\{p_1,p_i,p_j\}$ with
$1<i<j\le m-n+2$ is collinear.
\end{Rmk}

\begin{Thm}\label{genus-zero}
Let $\tS$ be the blow-up of the $n$-th Hirzebruch surface $\Sigma_n$ at $m$ points in general position.
Then the tangent bundle $T_{\tS}$ is big if and only if
\[
m\leq
\begin{cases}
3 & \text{if $n=0,\,1$},\\
n+1 & \text{otherwise}.
\end{cases}
\]
\end{Thm}

\begin{proof}
For $n=0,\,1$, the Hirzebruch surface $\Sigma_n$ is a del Pezzo surface of degree $8$.
Indeed, $\Sigma_0\simeq\PP^1\times\PP^1$ and $\Sigma_1\simeq \Bl_p\PP^2$.
Moreover, if $m\leq 7$ and the blown-up points are in general position, then the blow-up $\tS$ of $\Sigma_n$ at these $m$ points is again a del Pezzo surface of degree $8-m$.
By \cite{HLS22}, it follows that $T_{\tS}$ is big if and only if $m\leq 3$.

Now let $n\geq 2$.
We first prove the non-bigness assertion: if $m\geq n+2$, then $T_{\tS}$ is not big.
By Lemma~\ref{blow-up}, it suffices to prove the assertion in the case $m=n+2$.
Let $\tS$ be the blow-up of $\Sigma_n$ at $n+2$ points in general position.
As described above, we may realize $\tS$ as $\Bl_{p_1,\ldots,p_4,q_1,\ldots,q_{n-1}}\PP^2$, where $p_1,\ldots,p_4$ are closed points of $\PP^2$ and the points $q_1,\ldots,q_{n-1}$ are infinitely near to $p_1$.
By Remark~\ref{general-position-condition}, $\Bl_{p_1,\ldots,p_4,q_1}\PP^2$ is a weak del Pezzo surface of degree~$4$ with either exactly one $(-2)$-curve or exactly two $(-2)$-curves and 8 lines.
If $\{p_2,p_3,p_4\}$ is not collinear, then $\Bl_{p_1,\ldots,p_4,q_1}\PP^2$ has only one $(-2)$-curve, otherwise, $\Bl_{p_1,\ldots,p_4,q_1}\PP^2$ has two $(-2)$-curves and $8$ lines.
By \cite[Theorem~1.2]{KKL24}, its tangent bundle is not big.
Since the tangent bundle of $\Bl_{p_1,\ldots,p_4,q_1}\PP^2$ is not big, Lemma~\ref{blow-up} implies that the tangent bundle of $\tS\simeq \Bl_{p_1,\ldots,p_4,q_1,\ldots,q_{n-1}}\PP^2$ is also not big.
Therefore, $T_{\tS}$ is not big.

We next prove the bigness assertion.
By Lemma~\ref{blow-up}, it suffices to consider the case $m=n+1$.
Let $\tS$ be the blow-up of $\Sigma_n$ at $n+1$ points in general position.
As before, we may realize $\tS$ as $\Bl_{p_1,p_2,p_3,q_1,\ldots,q_{n-1}}\PP^2$, where $p_1,p_2,p_3$ are closed points of $\PP^2$ and the points $q_1,\ldots,q_{n-1}$ are infinitely near to $p_1$.

Let $\Kk_i$ be the family of strict transforms $C_i$ of lines in $\PP^2$ passing through $p_i$ for $i=1,\,2,\,3$, and let $\Ll_j$ be the family of strict transforms $D_j$ of conics in $\PP^2$  passing through $p_1$, $p_2$, $p_3$, and $q_j$ for $j=1,2,\ldots,n-1$.
By Remark~3.3, the families \(\mathcal K_i\) and \(\mathcal L_j\) are pencils of lines and conics, respectively.
Then
\begin{align*}
C_1&\sim H-E_1-F_1-\cdots-F_{n-1},\\
C_2&\sim H-E_2,\\
C_3&\sim H-E_3,\\
D_j&\sim 2H-E_1-E_2-E_3-F_1-\cdots-2F_j-\cdots-F_{n-1}\quad
\text{for $j=1,\,2,\,\ldots,\,n-1$},
\end{align*}
where $H$ denotes the pullback of a line in $\PP^2$, $E_i$ denotes the strict transform of the exceptional divisor over $p_i$, and $F_j$ denotes the exceptional divisor over $q_j$.
By \eqref{conic-bundle-formula}, the total dual VMRTs $\breve\Cc_i$ and $\breve\Dd_j$ associated to $\Kk_i$ and $\Ll_j$, respectively, are given as follows for $i=1,2,3$ and $j=1,\ldots,n-1$.
\begin{align*}
\breve\Cc_1
&\sim \zeta-\Pi^*H-\Pi^*E_1+\Pi^*E_2+\Pi^*E_3\\
\breve\Cc_2
&\sim \zeta-\Pi^*H+\Pi^*E_1-\Pi^*E_2+\Pi^*E_3+2\Pi^*F_1+\cdots+2\Pi^*F_{n-1}\\
\breve\Cc_3
&\sim \zeta-\Pi^*H+\Pi^*E_1+\Pi^*E_2-\Pi^*E_3+2\Pi^*F_1+\cdots+2\Pi^*F_{n-1}\\
\breve\Dd_j
&\sim \zeta+\Pi^*H-\Pi^*E_1-\Pi^*E_2-\Pi^*E_3-2\Pi^*F_j\quad
\text{for $j=1,\,2,\,\ldots,\,n-1$}
\end{align*}
Then a positive multiple of $\zeta$ can be written as
\[
(4n-2)\zeta
\sim
\left(
2(n-1)\breve\Cc_1
+\breve\Cc_2
+\breve\Cc_3
+2\sum_{j=1}^{n-1}\breve\Dd_j
+(4n-6)\Pi^*E_1
\right)
+2\Pi^*H
\]
where the term in parentheses is an effective divisor and $2\Pi^*H$ is the pullback of a big divisor on $X$.
Therefore, $T_X$ is big by Lemma~\ref{big_criterion}.
\end{proof}

\subsection*{Acknowledgements}
The authors would like to express their sincere gratitude to Yongnam Lee and YongJoo Shin for their valuable and insightful discussions.

\end{document}